\documentclass{amsart}
    \usepackage{amsmath}
    \usepackage{amssymb,amsfonts}
    \usepackage[all,arc]{xy}
    \usepackage{caption}
    \usepackage{enumerate}
    \usepackage{mathrsfs}
    \usepackage{amsthm}
    \usepackage[shortlabels]{enumitem}
    \usepackage{extarrows}
    \usepackage{verbatim}
    \usepackage{dsfont}
    \usepackage{tikz-cd}
    \usepackage{tensor}
    \usepackage{mathtools}
    \usepackage{upgreek}
    \usepackage{xcolor}
    \usepackage{extarrows}
    \usepackage{mathtools}
    \usepackage{tikz}
    \usepackage{pgfplots}
    \usetikzlibrary{positioning}
    \usetikzlibrary{calc}
   \usepackage{ stmaryrd }

    \usepackage{hyperref}

    \newtheorem{thm}{Theorem}[section]
    \newtheorem*{main}{Main Theorem}
    \newtheorem{cor}[thm]{Corollary}
    \newtheorem{prop}[thm]{Proposition}
    \newtheorem{lem}[thm]{Lemma}

    \theoremstyle{definition}

    \theoremstyle{remark}
    
    \newtheorem{rem}[thm]{Remark}

    \newcommand{\Z}{\mathbb{Z}}
    \newcommand{\R}{\mathbb{R}}
    \newcommand{\C}{\mathbb{C}}

    \newcommand{\Ocal}{\mathcal{O}}

    \newcommand{\Ucal}{\mathcal{U}}

    \newcommand{\ubf}{\mathbf{u}}
    \newcommand{\vbf}{\mathbf{v}}
    \newcommand{\wbf}{\mathbf{w}}

    \newcommand{\Coh}{\mathrm{Coh}}

    \newcommand{\NS}{\mathrm{NS}}

    \newcommand{\Pic}{\mathrm{Pic}}

    \newcommand{\Mov}{\mathrm{Mov}}
    
    \newcommand{\Hilb}[1]{\mathrm{Hilb}^{#1}}
    \newcommand{\Db}{\mathrm{D}^b}
    \newcommand{\Stab}{\mathrm{Stab}}

    \makeatletter
    
    \newcommand*{\rom}[1]{\expandafter\@slowromancap\romannumeral #1@}
    
    \makeatletter
    
    \let\c@equation\c@thm
    \makeatother
    \numberwithin{equation}{section}

\title[Birational Geometry of Beauville-Mukai systems III]{Birational Geometry of Beauville-Mukai systems III: Asymptotic behavior }
\author{Xuqiang Qin and Justin Sawon}
\address{Department of Mathematics, University of North Carolina, Chapel Hill, NC 27599-3250, USA}
\email{qinx@unc.edu, sawon@email.unc.edu}

\date{October 2022}

\pgfplotsset{compat=1.17} 

\begin{document}
\begin{abstract}
Suppose that a Hilbert scheme of points on a K3 surface $S$ of Picard rank one admits a rational Lagrangian fibration. We show that if the degree of the surface is sufficiently large compared to the number of points, then the Hilbert scheme is the unique hyperk{\"a}hler manifold in its birational class. In particular, the Hilbert scheme is a Lagrangian fibration itself, which we realize as coming from a (twisted) Beauville-Mukai system on a Fourier-Mukai partner of $S$. We also show that when the degree of the surface is small our method can be used to find all birational models of the Hilbert scheme.
\end{abstract}

\maketitle
\section{Introduction}
The Hyperk{\"a}hler SYZ Conjecture~\cite{Huy03,Saw03} proposes necessary and sufficient conditions under which a compact hyperk{\"a}hler admits a rational Lagrangian fibration. In the case of the Hilbert scheme of $N+1$ points on a K3 surface $S$ of degree $2d$, it states that $\Hilb{N+1}(S)$ is birational to a Lagrangian fibration if and only if $Nd$ is a perfect square; this was verified in special cases in~\cite{Mar06,Saw07}, and ultimately proved in greater generality in~\cite{BM14b}.

In our previous papers~\cite{QS22a,QS22b} we investigated the cases when $N$ is of the form $dn^2$ for a positive integer $n$. In those cases $\Hilb{N+1}(S)$ is birational (but not isomorphic) to a \emph{Beauville-Mukai system} on $S$~\cite{Bea91,Muk84}, a certain moduli space of pure dimension one sheaves on $S$, which admits the structure of a Lagrangian fibration. Moreover, the birational map can be decomposed into a sequence of stratified Mukai flops, where each flop yields a new \emph{birational model}, a compact hyperk\"ahler manifold birational to $\Hilb{N+1}(S)$.

In this paper we investigate the general case. This includes the following situation which was investigate by Markushevich~\cite{Mar06} and the second author~\cite{Saw07}: instead of having $N/d$ be a perfect square as in the previous paragraph, they looked at the cases when $d/N$ is a perfect square.
\begin{thm}\cite{Mar06,Saw07}\label{MS}
Let $S$ be a K3 surface with $\Pic(S)=\Z\cdot H$ and $H^2=2Nm^2$, where $N$, $m\in \Z^+$ and $m\geq 2$. Then $\Hilb{N+1}(S)$ is a Lagrangian fibration.
\end{thm}
In these cases one can find a K3 surface $S'$, typically not isomorphic to $S$, that is derived equivalent to $S$. Using this derived equivalence one can show that ideal sheaves of length $N+1$ zero-dimensional subschemes of $S$ correspond to twisted pure dimension one sheaves on $S'$. The moduli space of such sheaves on $S'$ is a \emph{twisted Beauville-Mukai system} and has the structure of a Lagrangian fibration. In particular, $\Hilb{N+1}(S)$ is isomorphic, not just birational, to this twisted Beauville-Mukai system.

Our main theorem provides a generalization of Theorem \ref{MS}.
\begin{main}[=Theorem 3.2]
Let $S$ be a K3 surface with $\Pic(S)=\Z\cdot H$ and $H^2=2d$. Suppose that $N\in\Z^+$ and $d/N$ is the square of a sufficiently large rational number. Then $\Hilb{N+1}(S)$ is a Lagrangian fibration.
\end{main}
The meaning of `sufficiently large' will become evident during the proof of the Main Theorem; in particular, it includes the cases of Theorem~1.1, which correspond to $d/N=m^2$ being the square of an integer $m\geq 2$. For the proof, it suffices to show that the movable cone of $\Hilb{N+1}(S)$ has only one chamber, or equivalently, that there are no interior walls (see Section~2 for details). We use the observation that an interior wall in the movable cone would correspond to a semi-circle in the open half plane parametrizing Bridgeland stability conditions on $S$, and the semi-circle will intersect a chosen line. When $d/N$ is sufficiently large, we show that such an intersection cannot happen. 

In addition, the proof of the Main Theorem can be used in a different way. Let $N$ be given. We know from the Main Theorem that when $d$ is large there is no wall in the interior of the movable cone of $\Hilb{N+1}(S)$. By following the argument of the proof, one can actually find the minimal degree $d_0$ where the cone starts to have only one chamber, and one can find all of the walls in the movable cone of $\Hilb{N+1}(S)$ for smaller values of $d$. This provides an efficient algorithm for finding all birational models of $\Hilb{N+1}(S)$ as $d$ varies.

The paper is organized as follows. We review some preliminary material and the set up for the Main Theorem in Section~2. We prove the Main Theorem and recover Theorem~1.1 as a corollary in Section~3. We illustrate the application of the proof described in the previous paragraph in Section~4. Finally, we provide a link of $\Hilb{N+1}(S)$ to the twisted Beauville-Mukai system in Section~5.

\subsection*{Acknowledgement} The authors would like to thank Nicolas Addington and Emanuele Macr\`i for helpful discussions. The second author gratefully acknowledges support from the NSF, grants DMS-1555206 and DMS-2152130.

\section{Setting}
Let $(S,H)$ be a polarized K3 surface such that $\Pic(S)=\Z\cdot H$. Let $H^2=2d=2\Delta k^2$ where $\Delta$ and $k\in \Z^+$. We do not assume that $\Delta$ is square-free. For $h\in \Z^+$ with $\mathrm{gcd}(k,h)=1$, consider the Hilbert scheme of $N+1=\Delta h^2+1$ points $\Hilb{\Delta h^2+1}(S)$. Recall that the N\'eron-Severi group $\NS(\Hilb{\Delta h^2+1}(S))$ has a basis given by
\begin{align*}
    \tilde{H}=\theta(0,-1,0) \qquad B=\theta(-1,0,-\Delta h^2),
\end{align*}
where $\theta$ is the Mukai homomorphism (for example, see~\cite{Yos01}). By~\cite[Proposition~13.1(a)]{BM14b},
\begin{align*}
    \Mov(\Hilb{\Delta h^2+1}(S))=\left\langle \tilde{H}, \tilde{H}-\frac{k}{h}B\right\rangle.
\end{align*}
Since $\Delta^2h^2k^2$ is a perfect square, the same proposition shows that $\Hilb{\Delta h^2+1}(S)$ admits a rational Lagrangian fibration induced by the divisor $h\tilde{H}-kB$. By~\cite{HT19}, the following statements are equivalent:
\begin{itemize}
    \item $\Hilb{\Delta h^2+1}(S)$ is a Lagrangian fibration;
    \item the movable cone of $\Hilb{\Delta h^2+1}(S)$ has only one chamber;
    \item there is no other compact hyperk\"ahler manifold birational to $\Hilb{\Delta h^2+1}(S)$.
\end{itemize}

We use the notation and conventions of~\cite[Section 2]{QS22b} for Bridgeland stability. For $x\in \R$ and $y\in \R_{>0}$, we denote by $\sigma_{x,y}$ the pair $(\Coh^{xH}(S), Z_{xH,yH})$, where $\Coh^{xH}(S)$ is obtained by tilting the category of coherent sheaves $\Coh(S)$ and $Z_{xH,yH}$ is a group homomorphism from the numerical Grothendieck group $K_{\mathrm{num}}(S)$:
\begin{align*}
    Z_{xH,yH}:K_{\mathrm{num}}(S)&\to \C\\
    F&\mapsto (e^{xH+\sqrt{-1}yH},\vbf(F)).
\end{align*}
Here $\vbf(F)$ is the Mukai vector of $F$ and the pairing on the right hand side is the Mukai pairing. If $x$ and $y$ satisfy a mild condition~\cite[Lemma 6.2]{Bri08}, $\sigma_{x,y}$ is a Bridgeland stability condition. We will call the subset of $\{(x,y)\in\R^2\mid y>0\}$ consisting of such pairs of $(x,y)$ \emph{the $xy$-plane}. We refer the reader to \cite{BM14b,QS22b} for the wall and chamber structure on the $xy$-plane.

\section{Proof of the Main Theorem}
A wall in $\Mov(\Hilb{\Delta h^2+1}(S))$ corresponds to a ray $\tilde{H}-\Gamma B$, where $0\leqslant \Gamma \leqslant \frac{k}{h}$. By the argument of~\cite[Theorem 10.8]{BM14a} and~\cite[Proposition 4.5]{QS22a}, any \emph{interior wall}, that is a wall whose $\Gamma$ satisfies $0<\Gamma<\frac{k}{h}$, will correspond to a flopping wall for $\vbf=(1,0,-\Delta h^2)$ in the $xy$-plane of Bridgeland stability conditions. Our first lemma narrows the range of $\Gamma$.

\begin{lem}
Let $W$ be a wall in the interior of   $\Mov(\Hilb{\Delta h^2+1}(S))$. Then its corresponding $\Gamma$ satisfies $\frac{2k^2}{h^2+k^2+1}\leqslant \Gamma<\frac{k}{h}$.
\end{lem}
\begin{proof}
The wall $W$ corresponding to $\Gamma$ is given in the $xy$-plane by $\left(x+\frac{1}{\Gamma}\right)^2+y^2=\frac{1}{\Gamma^2}-\frac{h^2}{k^2}$. This semicircle intersects the vertical line $x=-1$ at $y_0^2=\frac{2}{\Gamma}-\frac{h^2+k^2}{k^2}$. Suppose that $\Gamma$ is less than $\frac{2k^2}{h^2+k^2+1}$. Then $y_0^2>\frac{1}{k^2}$. Recall that $\sigma_{-1,y}$ is a stability condition if $y^2>\frac{1}{k^2}$ by~\cite[Lemma~6.2]{Bri08} or~\cite[Lemma~4.2]{QS22b}. Thus $\sigma_{-1,y_0}$ is a stability condition and $W$ induces at $\sigma_{-1,y_0}$ a decomposition of Mukai vectors:
\begin{align*}
    (1,0,-\Delta h^2)=(a,b,c)+(1-a,-b,-\Delta h^2-c)
\end{align*}
where $a$, $b$, $c\in \Z$ are integers and the two Mukai vectors on the right are those of objects in $\Coh^{-1}(S)$. As a result, $0\leqslant b+a\leqslant 1$. We have
\begin{align*}
    &Z_{-1,y_0}(1,0,-\Delta h^2)=\Delta (k^2y_0^2-k^2+h^2)+\sqrt{-1}(2\Delta k^2y_0);\\
    &Z_{-1,y_0}(a,b,c)=(-2\Delta k^2b-c+a\Delta k^2(y_0^2-1))+\sqrt{-1}(2\Delta k^2y_0(b+a)).
\end{align*}
Since the above two complex numbers have the same argument, $b+a>0$. A similar argument applied to $(1-a,-b,-\Delta h^2-c)$ shows that $b+a<1$. Thus $0<b+a<1$, which contradicts the assumption that $a$, $b\in \Z$.
\end{proof}

\begin{thm}\label{asymp}
Let $\Delta$ and $h\in \Z^+$ be fixed. For $k\in\Z^+$ sufficiently large with $\mathrm{gcd}(h,k)=1$, the movable cone of $\Hilb{\Delta h^2+1}(S)$ has only one chamber. In particular:
\begin{itemize}
    \item the divisor $h\tilde{H}-kB$ induces a Lagrangian fibration on $\Hilb{\Delta h^2+1}(S)$;
    \item there is no other hyperk\"ahler manifold birational to $\Hilb{\Delta h^2+1}(S)$.
\end{itemize}
\end{thm}
\begin{rem}
We shall see in the proof that $k\geqslant\Delta h(h-1)^2+\frac{3h}{2}$ is always `sufficiently large', though the theorem is often true for smaller values of $k$ too.
\end{rem}
\begin{proof}
It suffices to show that for large $k$ there are no interior walls in the movable cone. Equivalently, it is enough to show that there are no flopping walls for $\vbf$ in the $xy$-plane. Suppose otherwise, i.e., that there is a flopping wall $W$ corresponding to the ray $\tilde{H}-\Gamma B$ where $\frac{2k^2}{h^2+k^2+1}\leqslant \Gamma< \frac{k}{h}$. Now $W$ is given in the $xy$-plane by $\left(x+\frac{1}{\Gamma}\right)^2+y^2=\frac{1}{\Gamma^2}-\frac{h^2}{k^2}$. Near the intersection of $W$ and $x=-\frac{1}{\Gamma}$, $W$ induces a decomposition of Mukai vectors:
\begin{align*}
    (1,0,-\Delta h^2)=(a,b,c)+(1-a,-b,-\Delta h^2-c)
\end{align*}
where $a$, $b$, $c\in \Z$ are integers. Moreover, by the definition of $\Coh^{-\frac{1}{\Gamma}}(S)$, $b+\frac{1}{\Gamma}a\geqslant 0$ and $(-b)+\frac{1}{\Gamma}(1-a)\geqslant0$.

Write $\wbf:=(a,b,c)$. We can assume that $\wbf$ is primitive, for otherwise we could replace $\wbf$ by its primitive scalar multiple and work instead with the resulting decomposition of Mukai vectors. By \cite[Theorem 5.7]{BM14b} and \cite[Section 9]{BM14b}, we have either that $\wbf^2=-2$ and $-\Delta h^2\leqslant(\wbf,\vbf)\leqslant \Delta h^2$ or that both $\wbf$ and $\vbf-\wbf$ are positive classes. In the second case, using the fact that $\vbf$ and $\wbf$ generate a hyperbolic lattice, we can assume that $0\leqslant \wbf^2<\frac{\Delta h^2}{2}$ and $2\wbf^2+1\leqslant (\wbf,\vbf)\leqslant \Delta h^2$. 

We have
\begin{align*}
    (\wbf,\vbf)=\Delta h^2a-c=:j,\\
    \frac{\wbf^2}{2}=\Delta k^2b^2-ac.
\end{align*}
Note that in both cases we have $-\Delta h^2\leqslant j\leqslant \Delta h^2$. Moreover, since the wall $\wbf$ corresponds to $\Gamma$, we have $\left(\wbf,\left(1,\frac{-1}{\Gamma},\Delta h^2\right)\right)=0$, and thus 
\begin{align}\label{Gamma}
    \Gamma=\frac{-2\Delta k^2b}{2\Delta h^2 a-j}.
\end{align}
We notice that since $\Gamma>0$, we must have $b\neq0$ and $ab\leqslant0$. We proceed by cases.

{\bf Case 1 : $a=0$.} Then $\Delta k^2b^2=\frac{\wbf^2}{2}<\frac{\Delta h^2}{4}$, which implies that $k|b|< \frac{h}{2}$. This cannot happen if $k\geqslant\frac{h}{2}$.

{\bf Case 2 : $a=1$ and $j\geqslant 0$.} Then $c=\Delta h^2-j$ and $\Delta k^2b^2=ac+\frac{\wbf^2}{2}=\Delta h^2-j+\frac{\wbf^2}{2}$. Combining the bounds for $j$ and $\wbf^2$, we obtain $-\frac{1}{\Delta}\leqslant k^2b^2< \frac{5h^2}{4}$. As a result $k(-b)< \frac{\sqrt{5}h}{2}$. This cannot happen if $k\geqslant\frac{\sqrt{5}h}{2}$.

{\bf Case 3 : $a=-1$ and $j\geqslant0$.} Then $c=-\Delta h^2-j$ and  $\Delta k^2b^2=ac+\frac{\wbf^2}{2}=\Delta h^2+j+\frac{\wbf^2}{2}$. Combining the bounds for $j$ and $\wbf^2$, we obtain $h^2-\frac{1}{\Delta}\leqslant k^2b^2< \frac{9h^2}{4}$. As a result $kb< \frac{3h}{2}$. This cannot happen if $k\geqslant\frac{3h}{2}$.\ 

{\bf Case 4 : $a=1$ and $j<0$.} Then one can replace $\wbf$ with $-\wbf$, and apply the argument of case 3.

{\bf Case 5 : $a=-1$ and $j>0$.} Similarly, one can replace $\wbf$ with $-\wbf$, and apply the argument of case 2.

{\bf Case 6 : $a<-1$.} Using $\Gamma< \frac{k}{h}$ and (\ref{Gamma}), we obtain $2\Delta h(ha+kb)< j\leqslant \Delta h^2$. On the other hand, we have $ha+kb=k\left(\frac{h}{k}a+b\right)> k\left(\frac{1}{\Gamma}a+b\right)\geqslant 0$. Thus $$1\leqslant ha+kb< \frac{h}{2}.$$
This already implies that the present case can be eliminated if $h=1$ or $2$. Let $i=ha+kb$. Since $\Delta k^2b^2=ac+\frac{\wbf^2}{2}$, we obtain that 
\begin{align*}
    -\frac{\Delta h^2}{4}-1<(j-2\Delta hi)a=\frac{\wbf^2}{2}-\Delta i^2<\frac{\Delta h^2}{4}-\Delta.
\end{align*}
We claim that $j\neq 2\Delta hi$. Otherwise, we would have $\wbf^2=2\Delta i^2$, $(\wbf,\vbf)=j=2\Delta hi$, and $\vbf^2=2\Delta h^2$, which contradicts our assumption that $\wbf$ and $\vbf$ generate a hyperbolic lattice. As a result, we obtain $|a|< \frac{\Delta h^2}{4}+1$. This implies that
\begin{align*}
k|b|=kb=h(-a)+i< \frac{\Delta h^3}{4}+h+\frac{h}{2}=\frac{\Delta h^3}{4}+\frac{3h}{2}.
\end{align*}
This cannot happen if $k\geqslant\frac{\Delta h^3}{4}+\frac{3h}{2}$.

{\bf Case 7 : $a>1$.} Using $\Gamma< \frac{k}{h}$ and (\ref{Gamma}), we obtain $2\Delta h(ha+kb)> j$. On the other hand, $(-b)+\frac{h}{k}(1-a)> (-b)+\frac{1}{\Gamma}(1-a)\geqslant 0$ implies that $ha+kb< h$. Thus 
\begin{align*}
    -\frac{h}{2}\leqslant \frac{j}{2\Delta h}< ha+kb\leq h-1.
\end{align*}
Let $i=ha+kb$. The computation in the previous case leads to
\begin{align*}
    -\Delta (h-1)^2-1\leqslant(j-2\Delta hi)a=\frac{\wbf^2}{2}-\Delta i^2<\frac{\Delta h^2}{4}.
\end{align*}
Since $j\neq 2\Delta hi$, we obtain $|a|\leqslant \max{\left(\Delta (h-1)^2+1,\frac{\Delta h^2}{4}\right)}$.  If $h=1$, then $-\frac{1}{2}\leqslant\frac{j}{2\Delta }< a+kb\leqslant0$, which would lead to a contradiction unless $j<0$. When $j<0$, we must have $i=a+kb=0$ and $\wbf^2=-2$. As a result, we obtain $ja=(j-2\Delta hi)a=\frac{\wbf^2}{2}-\Delta i^2=-1$, which contradicts our assumption $a>1$. If $h\geq 2$, then $\Delta (h-1)^2+1>\frac{\Delta h^2}{4}$ for any $\Delta\in\Z^+$. Hence $|a|\leq \Delta (h-1)^2+1$. Then
\begin{align*}
k|b|=k(-b)=ha-i< \Delta h(h-1)^2+h+\frac{h}{2}=\Delta h(h-1)^2+\frac{3h}{2}.
\end{align*}
This cannot happen if $k\geqslant\Delta h(h-1)^2+\frac{3h}{2}$.

Since $\Delta h(h-1)^2+\frac{3h}{2}\geqslant\frac{\Delta h^3}{4}+\frac{3h}{2}$ when $h\geqslant2$, we conclude that if $k\geqslant\Delta h(h-1)^2+\frac{3h}{2}$ then there is no flopping wall in the movable cone of $\Hilb{\Delta h^2+1}(S)$. This completes the proof.
\end{proof}

We can now recover Theorem \ref{MS} as a corollary.
\begin{cor}\cite{Mar06,Saw07,BM14a}\label{cor}
Let $S$ be a K3 surface with $\Pic(S)=\Z\cdot H$ and $H^2=2\Delta k^2$, where $\Delta$, $k\in \Z^+$ and $k\geqslant2$. As above, we do not assume that $\Delta$ is square-free. Then the movable cone of $\Hilb{\Delta+1}(S)$ has one chamber and $\Hilb{\Delta+1}(S)$ is a Lagrangian fibration.
\end{cor}
\begin{proof}
In this case we have $h=1$. The statement follows since
$$ k\geqslant2>\frac{3}{2}= \Delta h(h-1)^2+\frac{3h}{2}.$$
\end{proof}
\begin{rem}
This is a (rare) case where our bound on $k$ is sharp. If $k=1$ then $\Hilb{\Delta+1}(S)$ is only birational to a Lagrangian fibration.
\end{rem}

\section{Applications of the Proof}

In specific cases, the general bound $k\geqslant\Delta h(h-1)^2+\frac{3h}{2}$ from the end of proof of Theorem~\ref{asymp} is usually too restrictive. Knowing the values of $\Delta$ and $h$, one can often obtain a much better bound by going through the argument of the proof of Theorem \ref{asymp} with the values of $\Delta$ and $h$ substituted in. Moreover, we obtain through this process all of the walls for small values of $k$. We demonstrate this in the following examples.

\subsection{The case $\Delta=h=1$}
The next proposition is well-known: the first part is the original example of a Mukai flop~\cite[Example~0.6]{Muk84}, and the second part follows from Corollary~\ref{cor}. Nonetheless, we demonstrate how the proof of Theorem~\ref{asymp} may be used to reproduce this result.
\begin{prop}
Let $S$ be a K3 surface such that $\Pic(S)=\Z\cdot H$ and $H^2=2k^2$, where $k\in \Z^+$. Then:
\begin{enumerate}
    \item if $k=1$ then the movable cone of $\Hilb{2}(S)$ has two chambers, corresponding to $\Hilb{2}(S)$ and the Beauville-Mukai system $M(0,1,-1)$;
    \item if $k\geqslant 2$ then the movable cone of $\Hilb{2}(S)$ has only one chamber. 
\end{enumerate}
\end{prop}
\begin{proof}
We adapt the proof of Theorem~\ref{asymp}. By the restrictions on $\wbf^2$, it must be either $-2$ or $0$. By the bounds on $(\wbf,\vbf)$, if $\wbf^2=0$ then $(\wbf,\vbf)=1$. But then the wall will induce a divisorial contraction. Thus $\wbf^2=-2$.

We have $a\neq 0$ since $k\geqslant1> \frac{h}{2}$. Since $h=1$, we have $|a|\leq 1$ by the proof of the theorem. Hence $a=\pm1$.

Suppose $a=1$ and $j\geqslant 0$. Then $k(-b)\leqslant \frac{\sqrt{5}}{2}$ and we have $k=1$ and $b=-1$. Now $k^2b^2=1=h^2-j+\frac{\wbf^2}{2}$ implies $j=-1$ and $c=2$. This contradicts our assumption that $j\geqslant0$.

Suppose $a=-1$ and $j\geqslant 0$. Then $kb\leqslant \frac{3}{2}$ and we have $k=1$ and $b=1$. Now $k^2b^2=1=h^2+j+\frac{\wbf^2}{2}$ implies $j=1$ and $c=-2$. This recovers the flop between $\Hilb{2}(S)$ and the Beauville-Mukai system $M(0,1,-1)$ on a degree two K3 surface.

The walls with $a=\pm1$ and $j<0$ have already been found in the previous two steps by replacing $\wbf$ with $-\wbf$. 
\end{proof}

\begin{rem}
This proposition also follows from~\cite[Lemma~13.3]{BM14b}, which states that the movable cone of $\Hilb{2}(S)$ has only one chamber if and only if the Pell's equation $X^2-4dY^2=5$ has no integer solutions. In our case
\begin{align*}
X^2-4dY^2=X^2-4k^2Y^2=(X-2kY)(X+2kY).
\end{align*}
We can assume that $X$ and $Y$ are non-negative, and hence we must have $X-2kY=1$ and $X+2kY=5$. This gives $Y=\frac{1}{k}$, a contradiction unless $k=1$. Our approach provides a method of finding the walls in the movable cone without needing to solve Pell's equation, or any other diophantine equation, directly.
\end{rem}

\subsection{The case $\Delta=1$ and $h=2$}
\begin{prop}
Let $S$ be a K3 surface such that $\Pic(S)=\Z\cdot H$ and $H^2=2k^2$, where $k$ is a positive odd integer. Then:
\begin{enumerate}
    \item if $k=1$ then the movable cone $\Hilb{5}(S)$ has five chambers;
    \item if $k=3$ then the movable cone $\Hilb{5}(S)$ has two chambers;
    \item if $k\geqslant 5$ then the movable cone of $\Hilb{5}(S)$ has only one chamber. 
\end{enumerate}
\end{prop}
\begin{proof}
When $k=1$, $\Hilb{5}(S)$ is birational to the Beauville-Mukai system $M(0,2,-1)$. Its birational geometry has been studied in~\cite{Hel20} and~\cite{QS22b}. This proves (1). 

For $k\geqslant 3$, we adapt the proof of Theorem~\ref{asymp}. Since $h=2$, by the restrictions on $\wbf^2$, we have either $\wbf^2=-2$ or $0$.

Since $k\geqslant 3=\frac{3h}{2}$, we can eliminate the case $a= 0$, the case $a=1$ and $j\geqslant 0$, and the case $a=-1$ and $j\geqslant 0$. Furthermore, there is no wall with $a=\pm1$ and $j<0$, as can be seen by replacing $\wbf$ with $-\wbf$.
 
Suppose that $a<-1$. Then $1\leqslant i=2a+kb<\frac{h}{2}=1$ leads to a contradiction.
 
Suppose that $a>1$. Then $-1< i= 2a+kb\leqslant 1$. If $i=0$, then $\frac{j}{4\Delta}<i=0$ and $ja=\frac{\wbf^2}{2}$, which is not possible by the assumption $a>1$ and the fact that $j$ can only be negative when $\wbf^2=-2$ (this argument will be repeatedly used to eliminate cases with $i=0$ and $|a|>1$). If $i=1$, then $(j-4)a=\frac{\wbf^2}{2}-1$. Since $a>1$, we must have $\wbf^2=-2$, $a=2$, and $j=3$. Then $k(-b)=2a-i=3$, and we obtain $k=3$ and $b=-1$. Finally, we obtain $c=5$.
 
To conclude, for $k\geqslant 3$, we only have a potential wall when $k=3$, and this potential wall is given by $\wbf=(2,-1,5)$. It is easy to check that it is an actual flopping wall. This prove parts (2) and (3) of the proposition.
\end{proof}

\subsection{The case $\Delta=1$ and $h=3$}
\begin{prop}
Let $S$ be a K3 surface such that $\Pic(S)=\Z\cdot H$ and $H^2=2k^2$, where $k\in \Z^+$ is not divisible by $3$. Then:
\begin{enumerate}
    \item if $k=1$ then the movable cone $\Hilb{10}(S)$ has $11$ chambers;
    \item if $k=2$, $4$, $5$, or $7$ then the movable cone $\Hilb{10}(S)$ has $5$, $3$, $2$, and $2$ chambers, respectively;
    \item if $k\geqslant 8$ then the movable cone of $\Hilb{10}(S)$ has only one chamber. 
\end{enumerate}
\end{prop}
\begin{proof}
When $k=1$, $\Hilb{10}(S)$ is birational to the Beauville-Mukai system $M(0,3,-1)$. Its birational geometry has been studied in \cite{QS22a}. This proves (1).

For $k\geqslant 2$, we adapt the proof of Theorem \ref{asymp}. Since $h=3$, by the restrictions on $\wbf^2$, we have either $\wbf^2=-2$, $0$, or $2$. Since $k\geqslant 2>\frac{h}{2}$, we can eliminate the case $a= 0$. 
 
Suppose that $a=1$ and $j\geqslant 0$. Then $k(-b)<\frac{3\sqrt{5}}{2}\approx 3.35$. We must have $k=2$ and $b=-1$. Using $c=9-j$ and $j=5+\frac{\wbf^2}{2}$, we obtain either
 \begin{itemize}
     \item $k=2$, $\wbf=(1,-1,5)$, and $j=4$,
     \item or $k=2$, $\wbf=(1,-1,4)$, and $j=5$.
 \end{itemize}
Note that $\wbf^2=2$ leads to $j=6$, in which case the lattice generated by $\vbf$ and $\wbf$ is not hyperbolic.
 
Suppose that $a=-1$ and $j\geqslant 0$. Then $kb\leqslant \frac92$. We must have either $(k,b)=(2,2)$ or $(4,1)$. Using $c=-9-j$ and $j=7-\frac{\wbf^2}{2}$, we obtain either
 \begin{itemize}
    \item $k=2$, $\wbf=(-1,2,-17)$, and $j=8$,
    \item $k=2$, $\wbf=(-1,2,-16)$, and $j=7$,
    \item $k=4$, $\wbf=(-1,1,-17)$, and $j=8$,
    \item or $k=4$, $\wbf=(-1,1,-16)$, and $j=7$.
 \end{itemize}
 
The walls with $a=\pm1$ and $j<0$ have already been found in the previous two steps by replacing $\wbf$ with $-\wbf$.
 
Suppose that $a<-1$. Then $1\leqslant i=3a+kb< \frac32$ implies $i=1$. Then $(j-6)a=\frac{\wbf^2}{2}-1$, which is equal to either $-2$, $-1$, or $0$. Since $a<-1$ and $j\neq 2hi=6$ (see proof of Theorem \ref{asymp}), we must have $a=-2$ and $j=7$. This gives
 \begin{itemize}
     \item $k=7$, $\wbf=(-2,1,-25)$, and $j=7$.
 \end{itemize}
 
Suppose that $a>1$. Then $-\frac32< i= 3a+kb\leqslant 2$. If $i=-1$ then $j\leqslant 2hi=-6$, and thus $\wbf^2=-2$. We have $(j+6)a=-2$, which implies $a=2$ and $j=-7$. A simple computation reveals that this is the previous case with $\wbf$ changed to $-\wbf$. As before $i$ cannot equal $0$. If $i=1$ then $(j-6)a=\frac{\wbf^2}{2}-1$. Similarly to the $a<-1$ case, we see that $a=2$ and $j=5$. This gives
 \begin{itemize}
     \item $k=5$, $\wbf=(2,-1,13)$, and $j=5$.
 \end{itemize}
If $i=2$ then $$(j-6i)a\leqslant (9-12)2=-6<-1-2^2\leqslant \frac{\wbf^2}{2}-i^2,$$ which contradicts the fact that $(j-6i)a=\frac{\wbf^2}{2}-i$.
 
To conclude, one readily checks that all of the potential walls above are actually flopping walls. Counting the number of chambers for each $k$ completes the proof of parts (2) and (3).
\end{proof}

\section{Description of the Lagrangian fibration of $\Hilb{\Delta h^2+1}(S)$}
In this section, we assume that $\Hilb{\Delta h^2+1}(S)$ is a Lagrangian fibration as in the Main Theorem. We will give an interpretation of the fibration using a Fourier-Mukai partner of $S$. Let $\ubf=(k,-h,\Delta kh^2)$. Then $\ubf^2=0$ and since $\mathrm{gcd}(k,h)=1$, $\ubf$ is primitive. Thus $S':=M_H(\ubf)$ is a smooth K3 surface~\cite{Muk87}. Let $\Ucal$ be the $\alpha$-twisted universal sheaf on $S\times M_H(\ubf)$, where $\alpha\in H^2(S',\Ocal_{S'}^*)$ is the gerbe obstructing the existence of a universal sheaf~\cite[Theorem~5.3.1]{Cal00}. Note that $\alpha$ is $k$-torsion, as $k$ is the gcd of $k$, $-hH.H=-2\Delta k^2h$, and $\Delta kh^2$. Let $\Phi:\Db(S)\to \Db(S',\alpha)$ be the Fourier-Mukai transform with kernel $\Ucal$, which is a derived equivalence~\cite[Theorem~5.5.1]{Cal00}.
\begin{lem}
Let $A$ and $B$ be integers such that $Bh-Ak=1$. Let $D=\theta_\ubf(B,-A,\Delta(-h+Akh))$. Then $\NS(S')\cong\Z\cdot D$. In particular, $S'$ is a K3 surface of degree $2\Delta$.
\end{lem}
\begin{proof}
By \cite[Theorem~1.5]{Muk87}, the Mukai homomorphism $\theta_\ubf:\ubf^\perp/\ubf\to \NS(S')$ is an isomorphism (here we use $\ubf^{\perp}$ to denote the orthogonal complement to $\ubf$ in the $(1,1)$ component of the Mukai lattice). Since $S'$ is a K3 surface, $\NS(S')$ is a free abelian group, hence isomorphic to $\Z$. It is easy to see that $(B,-A,\Delta(-h+Akh))$ is in $\ubf^\perp$ but not a multiple of $\ubf$. It remains to show that $(B,-A,\Delta(-h+Akh))$ is primitive in $\ubf^\perp/\ubf$. Suppose that it is $N$-divisible, for $N\in\Z_{>0}$. Then there exist $L\in\Z$ such that $(B,-A,\Delta(-h+Akh))+L\ubf$ is $N$ times a vector in $\ubf^\perp$. In particular, $N$ divides both $B+Lk$ and $-A-Lh$. It follows that $N$ divides $h(B+Lk)+k(-A-Lh)=1$, and hence that $N=1$. Thus $D$ is a generator of $\NS(S')$ and the degree of $S'$ follows from computing $D^2$.
\end{proof}

We use $H'$ to denote the ample generator of $\NS(S')$; $H'$ will equal $\pm D$, but the sign is not important. Let $\sigma\in \Stab^\dagger(S)$ be a generic Bridgeland stability condition on $S$ such that $\Hilb{\Delta h^2+1}(S)=M_\sigma(\vbf)$, where $\vbf=(1,0,-\Delta h^2)$. The equivalence $\Phi$ induces an isomorphism
\begin{align*}
    M_\sigma(\vbf)\cong M_{\Phi_*(\sigma)}(\Phi_*(\vbf))
\end{align*}
where $\Phi_*(\sigma)$ is a Bridgeland stability condition and $\Phi_*(\vbf)$ is a Mukai vector on $S'$. Since $\Phi_*(\ubf)=(0,0,1)$ and $(\vbf,\ubf)=0$, we see that $\Phi_*(\vbf)$ has rank zero. By~\cite[Lemma 11.2]{BM14b}, up to composing $\Phi$ with a shift, we can assume that $\Phi_*(\vbf)=(0,C,s)$ where $C>0$. In fact, $C$ must be a multiple of $H'$, and since $(H')^2=D^2=2\Delta$ and $C^2=(0,C,s)^2=\vbf^2=2\Delta h^2$, we find that $C=hH'$. The moduli space $M_{H'}((0,C,s),\alpha)$ is known as the twisted Beauville-Mukai system and it admits the  structure of a Lagrangian fibration~\cite[Lemma 11.3]{BM14b}. Now we see that
\begin{align*}
    \Hilb{\Delta h^2+1}(S)=M_\sigma(\vbf)\cong M_{\Phi_*(\sigma)}(\Phi_*(\vbf))\cong M_{H'}((0,C,s),\alpha),
\end{align*}
where the last isomorphism follows from the fact that $M_{H'}((0,C,s),\alpha)$ is birational to $M_{\Phi_*(\sigma)}(\Phi_*(\vbf))$, which has no other birational models (recall that we are assuming we are in the situation of the Main Theorem, i.e., $k$ is sufficiently large). This gives a description of the Lagrangian fibration structure on $\Hilb{\Delta h^2+1}(S)$.

\begin{rem}
In~\cite{Saw07}, the second author proved (with the help of Yoshioka) that if $h=1$ then in fact the Fourier-Mukai transform $\Phi$ with kernel $\Ucal$ induces an isomorphism 
   \begin{align*}
       \Hilb{\Delta h^2+1}(S)\xrightarrow{\cong} M_{H'}((0,C,s),\alpha)
   \end{align*}
up to a shift. The assumption $h=1$ played a crucial role in the argument of~\cite{Saw07}, and we were unable to generalize this statement to $h\geq2$. Specifically, although $\Hilb{\Delta h^2+1}(S)$ and $M_{H'}((0,C,s),\alpha)$ are isomorphic, we were unable to show that $\Phi$ itself induces an isomorphism. We proved that in going from $M_{\Phi_*(\sigma)}(\Phi_*(\vbf))$ to $M_{H'}((0,C,s),\alpha)$ we do not cross any {\em flopping\/} walls, but we might still cross walls that induce isomorphisms. So it is possible that $\Phi$ does not take elements of $\Hilb{\Delta h^2+1}(S)$ to elements of $M_{H'}((0,C,s),\alpha)$.
\end{rem}

\end{document}